\documentclass[11pt]{article} % use larger type; default would be 10pt

\usepackage{geometry} % to change the page dimensions
\usepackage{mathrsfs} %script font package
\usepackage{amsthm} %theorem layout package
\usepackage{amsmath} % for matrices
\usepackage{amsfonts}

\usepackage{nicefrac}

\usepackage[pdftex]{graphicx}
\usepackage{color}

\allowdisplaybreaks

\numberwithin{equation}{section}

\newcommand{\NN}{\mathbb{N}}
\newcommand{\ZZ}{\mathbb{Z}}
\newcommand{\RR}{\mathbb{R}}
\newcommand{\CC}{\mathbb{C}}
\newcommand{\ZZp}{\ZZ[i]^+}
\newcommand{\FQ}{\mathcal{F}_Q}
\newcommand{\FIQ}{\mathcal{F}_{I,Q}}
\newcommand{\GS}{\mathcal{G}_S}
\newcommand*\diff{\mathop{}\!\mathrm{d}}

\newcommand{\bigoh}{\mathcal{O}}
\newcommand{\bgoh}[1]{\bigoh\left( {#1} \right)}

\newcommand{\RE}[1]{\operatorname{Re} \left({#1}\right)}
\newcommand{\IM}[1]{\operatorname{Im} \left({#1}\right)}

\newcommand{\abs}[1]{\lvert#1\rvert}

\newcommand{\ifotherdef}[4]
{
	\left\{
		\begin{array}{ll}
			#1 & \mbox{if } #2 \\
			#3 & \mbox{otherwise. } #4
		\end{array}
	\right.
}

\makeatletter
\def\blfootnote{\xdef\@thefnmark{}\@footnotetext}
\makeatother

\newtheorem{defn}{Definition}[section]
\newtheorem{prop}{Proposition}[section]
\newtheorem{thm}{Theorem}[section]
\newtheorem{lem}{Lemma}[section]

\newtheorem{fact}{Fact}[section]

\title{First Moment of Distances Between Centres of Ford Spheres}

\date{May 2018}

\author{Kayleigh Measures    \\ {\small\sc (York) }    }

\begin{document}

\maketitle

\maketitle

\begin{abstract}
This paper aims to develop the theory of Ford spheres in line with the current theory for Ford circles laid out in a recent paper by S. Chaubey, A. Malik and A. Zaharescu.  As a first step towards this goal, we establish an asymptotic estimate for the first moment  \begin{equation*}\centering \mathcal{M}_{1,I_2} (S) = \sum\limits_{\substack{\frac{r}{s},\frac{r'}{s'}\in \GS \\ consec}} \frac{1}{2\abs{s}^2} + \frac{1}{2\abs{s'}^2}  \, , \end{equation*}
where the sum is taken over pairs of fractions associated  with `consecutive' Ford spheres of radius less than or equal to $\frac{1}{2S^2}$.
\end{abstract}

\renewcommand{\baselinestretch}{1}
\parskip=1ex


\section{Introduction and Motivation}

Ford spheres were first introduced by L. R. Ford in \cite{Ford} alongside their two dimensional analogues, Ford circles. For a Farey fraction $\frac{p}{q}$, its Ford circle is a circle in the upper half-plane of radius $\frac{1}{2q^2}$ which is tangent to the real line at $\frac{p}{q}$.  Similarly, for a fraction $\frac{r}{s}$ with $r$ and $s$ Gaussian integers, its Ford sphere is a sphere in the upper half-space of radius $\frac{1}{2\abs{s}^2}$ tangent to the complex plane at $\frac{r}{s}$.

This paper studies moments of distances between centres of Ford spheres, in line with the moment calculations for Ford circles produced by S. Chaubey, A. Malik and A. Zaharescu.  In \cite{Chau}, Chaubey et al. consider those Farey fractions in $\mathcal{F}_Q$ which lie in a fixed interval $I:=[\alpha,\beta] \subseteq [0,1]$ for rationals $\alpha$ and $\beta$.   They call the set of Ford circles corresponding to these fractions $\mathcal{F}_{I,Q}$, and its cardinality is denoted $N_I(Q)$.  The circles $C_{Q,j}$ ($1 \leq j \leq N_I(Q)$) in $\FIQ$ are ordered so that each circle is tangent to the next.  The center of the circle $C_{Q,j}$ is denoted by $O_{Q,j}$. They then consider the $k$-moments of the distances between the centres of consecutive circles for any positive integer $k$; namely the quantity
  \begin{equation*}
  \centering
    \mathcal{M}_{k,I}(Q) := \frac{1}{| I |} \sum_{j=1}^{N_I(Q)-1} (D( O_{Q,j}, O_{Q,j+1}))^k
  \end{equation*}
  where $D( O_{Q,j}, O_{Q,j+1})$ denotes the Euclidean distance between the centres  $O_{Q,j}$ and $O_{Q,j+1}$.
Finally, varying $Q$ and choosing a real variable $Y$ so that $Q=\lfloor Y \rfloor$, they consider the averages of these $k$-moments for all large $X$, defined by
\begin{equation*}
  \centering
    \mathcal{A}_{k,I}(X) := \frac{1}{X} \int_{X}^{2X} \mathcal{M}_{k,I}(Q)  \diff Y ,
\end{equation*}
which they find satisfy nice asymptotic formulas.  In particular, for a constant $B_1(I)$ depending only on the chosen interval, they show that
\begin{equation*}
  \mathcal{A}_{1,I}(X) = \dfrac{6}{\pi^2} \log (4X) + B_1(I) + O\left( \dfrac{\log X}{X} \right).
\end{equation*}

While the fundamental properties of Ford circles required for calculating these moments are well-established, surprisingly the same is not true for Ford spheres.  For example, the Farey fractions in the interval $[0,1]$ can be generated from 0 and 1 by taking mediants of consecutive fractions, but there is no similar established method for generating Gaussian rationals in the unit square of the complex plane starting from 0,1,$i$ and $1+i$.  Moreover, there is no existing notion of `consecutive' in the higher dimensional case.  Two Farey fractions are consecutive in $\FQ$ if, when the members of $\FQ$ are listed in increasing order of size, one immediately follows the other in the list.  So, for example, in $\mathcal{F}_3 = \{0,\frac{1}{3},\frac{1}{2},\frac{2}{3},1\}$, $\frac{1}{2}$ is consecutive to $\frac{1}{3}$. However, Gaussian rationals have no such natural ordering, so consecutivity for Ford spheres cannot be defined in the same way.  Instead we will examine consecutivity in the context of the Ford circles and use this to give meaning to `consecutive spheres' in Section \ref{sec:spheres}.  Further, in Section \ref{sec:gencff}, we describe a method for generating Gaussian rationals similar to that for Farey fractions using a variation on the mediant operation.

With the work of Chaubey et. al. and these new notions in mind,  it makes sense to define the $k^{th}$ moment for Ford spheres as the sum of the $k^{th}$ powers of the distances between the centres of consecutive spheres.  The definition in Section \ref{sec:spheres} will require consecutive spheres to be tangent, so the distance between their centres is given by the sum of their radii and the $k^{th}$ moment is thus defined as
\begin{equation}\label{eqn:moment1}
  \mathcal{M}_{k,I_2} (S) = \sum\limits_{\substack{\frac{r}{s},\frac{r'}{s'}\in \GS \\ consec}} \left( \frac{1}{2\abs{s}^2} + \frac{1}{2\abs{s'}^2} \right)^k
\end{equation}
where $I_2$ is the unit square in the upper right quadrant of the complex plane, and we sum over consecutive fractions lying in $I_2$.  As finding an asymptotic estimate for the first moment for Ford spheres is already sufficiently more difficult than it is for Ford circles, this paper will deal with the case $k=1$ and higher moments will be dealt with in a later paper.

In order to study the first moment for the Ford spheres, we first need to understand when two Gaussian integers appear as denominators of consecutive fractions.  We will also need to be able to count how many Gaussian integers $s'$ are denominators of fractions which are consecutive to fractions with a given denominator $s$.  Once this has been achieved, it will be shown that the first moment satisfies the following formula.

\begin{thm} \label{thm:mainthm}
For a given positive integer $S$ and $\mathcal{M}_{k,I_2} (S)$ as defined in \eqref{eqn:moment1} with $k=1$,
\begin{equation*}
  \mathcal{M}_{1,I_2}(S) = \pi\zeta_i^{-1}(2)(8C-1) S^2 + \bigoh_{\epsilon}(S^{1+\epsilon}).
\end{equation*}
where $C = - \int_{0}^{\frac{1}{\sqrt{2}}} \ln(\sqrt{2}u)(1-u^2)^{\frac{1}{2}} du $.
\end{thm}

To begin, Section 2 will cover relevant definitions and background material.  This will include results for Ford circles as well as Ford spheres, as understanding the properties of Ford circles will be necessary for formulating the equivalent properties of Ford spheres.  Section 3 contains the preliminary lemmas for Gaussian integers required in the proof of Theorem \ref{thm:mainthm}.  In Section 4 two lemmas necessary for calculating $\mathcal{M}_{k,I_2}(S)$ will be stated and proved. Finally, Theorem \ref{thm:mainthm} will then be proved in Section 5.

\section{Definitions}
In this section we review some relevant notions and facts about Farey fractions, Ford circles and Ford spheres.  In particular we see when two Farey fractions are called consecutive and then use this to give a suitable definition to the term for Ford spheres.

\subsection{Farey Fractions}

First we define $\mathcal{F}_Q$, the Farey fractions of order $Q$.
\begin{defn}For a positive integer $Q$,
\begin{equation*}
\centering
\mathcal{F}_Q = \{ \frac{p}{q} \in [0,1] : p,q \in \ZZ, (p,q)=1, q \leq Q \}.
\end{equation*}
\end{defn}
Here and in the following $(p,q)$ denotes the greatest common divisor of $p$ and $q$.  The fractions in this set are taken to be in order of increasing size.
\begin{defn}
Two rationals $\frac{a}{b} < \frac{c}{d}$ in $\mathcal{F}_Q$ are called adjacent if
\begin{equation}
bc-ad=1.
\end{equation}
They are consecutive in $\FQ$ if they are adjacent and $b+d>Q$.
\end{defn}

Note that this definition coincides with the usual meaning of consecutive, i.e. if $\frac{a}{b}$ and $\frac{c}{d}$ satisfy these conditions then $\frac{c}{d}$ will immediately follow $\frac{a}{b}$ in $\FQ$.  The Farey fractions can be constructed starting from 0 and 1 by taking mediants of consecutive Farey fractions.  This construction can be visualised in the left-hand side of the Stern-Brocot tree, shown in figure \ref{fig:sbtree}.  The lemma below follows from results in Chapter 3 of \cite{HardyWright} and says that all rational numbers will eventually be generated by this method.

\begin{lem}
Given coprime integers $0 \leq p < q$, $\frac{p}{q}$ occurs as a mediant of two fractions which are consecutive in $\mathcal{F}_{q-1}$.
\end{lem}

\begin{figure}
    \centering
    \includegraphics[width=12cm]{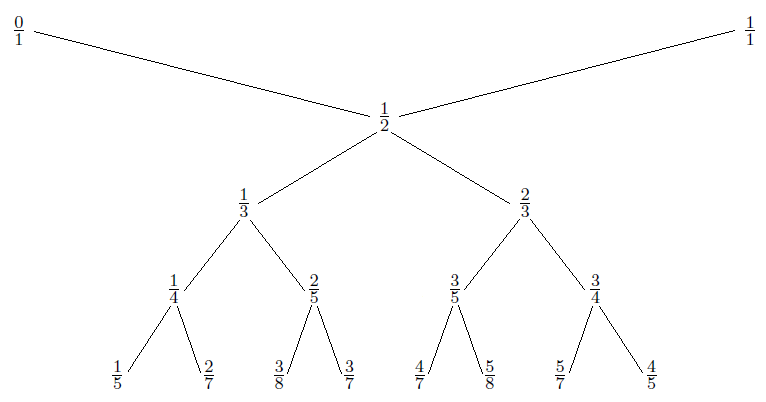}
    \caption{The Farey fractions as the left hand side of the Stern-Brocot tree.}
    \label{fig:sbtree}
\end{figure}

\subsection{Ford Circles}

Related to the Farey fractions are the Ford circles, which were introduced by L. R. Ford in \cite{Ford}.
\begin{defn}
For a Farey fraction $\frac{p}{q}$, its Ford circle is the circle in the upper half-plane of radius $\frac{1}{2q^2}$ which is tangent to the $x$-axis at $\frac{p}{q}$.
\end{defn}
This can be seen in Figure \ref{fig:fordcircles}.  Considering $\mathcal{F}_Q$ is equivalent to drawing the line $y=\frac{1}{2Q^2}$ and considering only rationals whose corresponding Ford circles have centres lying on or above this line.

\begin{figure}
    \centering
    \includegraphics[width=8cm]{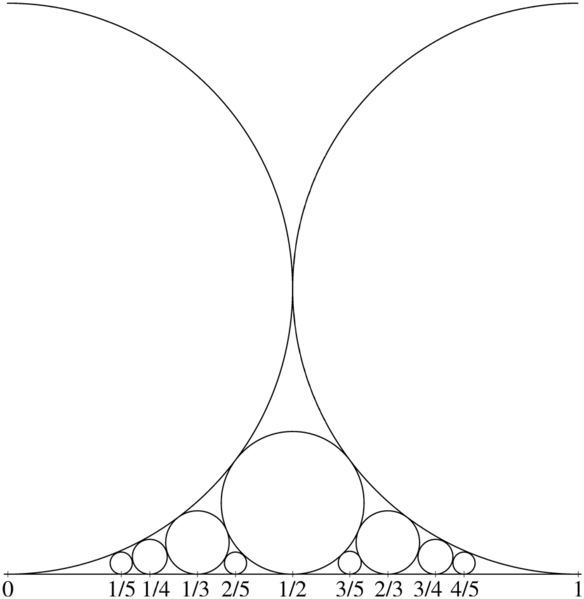}
    \caption{The Ford Circles in the interval [0,1].}
    \label{fig:fordcircles}
\end{figure}

When two Farey fractions of order $Q$ are adjacent their corresponding Ford circles will be tangent.  If the Farey fractions are consecutive, their Ford circles will be tangent and there will be no smaller circle between them for that order $Q$.  Following Chaubey et al.'s notation from \cite{Chau}, we denote the Ford circle corresponding to the $j$th member of $\mathcal{F}_Q$ by $C_{Q,j}$, and its centre by $O_{Q,j}$.  Since they are tangent, the distance between the centres of two consecutive Ford circles will be given by the sum of their radii,
\begin{equation}
\centering
  D(O_{Q,j} , O_{Q,j+1}) = \frac{1}{2q_j^2} + \frac{1}{2q_{j+1}^2}.
\end{equation}

\subsection{Ford Spheres} \label{sec:spheres}

In \cite{Ford}, Ford introduces a complex analogue of his circles, which we call Ford spheres.  Where the Ford circles are related to the Farey fractions, which lie in $\mathbb{Q}$, the Ford spheres are based on a complex analogue of these fractions, which take values in $\mathbb{Q}[i]$.  Then, in place of $\FQ$ we define$\GS$ as follows.
\begin{defn}\label{def:GS}  Let $S$ be a positive integer and let $u$ be a unit in $\ZZ[i]$,
\begin{equation*}
\centering
\GS := \left\{ \tfrac{r}{s} \in I_2 : r,s \in \mathbb{Z}[i], (r,s) = 1, |s| \leq S\right\}
\end{equation*}
where $I_2$ is the unit square in the upper right quadrant of the complex plane.
\end{defn}
In place of a circle of radius $\frac{1}{2q^2}$ placed on the $x$-axis at $\frac{p}{q} \in \mathcal{F}_Q$ , Ford considers a sphere of radius $\frac{1}{2|s|^2}$ in the upper half-space $\mathbb{C} \times \mathbb{R}^+$, touching the complex plane at $\frac{r}{s} \in \mathcal{G}_S$. Similarly to the Ford circles, two spheres with corresponding fractions $\frac{r}{s}$ and $\frac{r'}{s'}$ are either tangent or disjoint.
\begin{defn}
Two fractions $\frac{r}{s}$ and $\frac{r'}{s'}$ in $\GS$ are called adjacent if their spheres are tangent, i.e. if
\begin{equation}
|r's-rs'| = 1.
\end{equation}
\end{defn}
Ford states that every fraction $\frac{r}{s}$ has an adjacent fraction $\frac{r'}{s'}$, and that any fraction of the form
  \begin{equation*}
  \centering
    \frac{r'_n}{s'_n} = \frac{r'+nr}{s'+ns}
  \end{equation*}
will also be adjacent to $\frac{r}{s}$, where $n$ is any Gaussian integer.  Ford also notes that of those spheres which are tangent to the sphere at $\frac{r}{s}$, two, three or four of them will be larger than that sphere.

In $\mathcal{M}_{k,I}(Q)$ the sum is taken over consecutive circles, so in order to consider analogous moments for Ford spheres, we will first define what it means for two fractions to be consecutive in $\GS$.  In $\FQ$ this meant that one fraction immediately followed the other when listed according to size, but this is not helpful for $\GS$ as the Gaussian rationals have no such natural ordering.  Alternatively, consecutivity for fractions in $\FQ$ is equivalent to the fractions being adjacent and the sum of their denominators being greater than $Q$.  If we try this approach for $\GS$, the question becomes, how should we ``add'' the denominators and then compare the result to $S$?.  Instead, we look at the definition in terms of Ford circles.  It is easy to see that if two Ford circles are consecutive in $\FQ$ then they are tangent and there is no smaller circle between them.  We use this idea to give an equivalent definition in terms of the Ford spheres.

\begin{defn}
Let $\frac{r}{s}$ and $\frac{r'}{s'}$ be fractions in $\GS$ with spheres $R$ and $R'$ respectively.  They are consecutive in $\GS$ if they are adjacent and there is at least one other fraction in $\mathbb{Q}[i]$ with sphere of radius less than $\frac{1}{2S^2}$ which is tangent to both $R$ and $R'$.
\end{defn}

Note that having a sphere of radius less than $\frac{1}{2S^2}$ means that the fraction will not be in $\GS$ itself.  The structure of the spheres means that for any two given tangent spheres there will be multiple smaller spheres which are tangent to both (unlike Ford circles where there is only one such circle).  The definition above ensures that two spheres are considered consecutive until all of those smaller spheres are in $\GS$; this will be necessary for constructing $\GS$ later (Section \ref{sec:gencff}).

\section{Preliminary Lemmas - Gaussian Integers} \label{sec:prelimGI}

In this section we lay out some notation and results for Gaussian integers that will be needed later.  Although many of these are analogues of well known facts for arithmetical functions, for completeness their proofs are also included.

\subsection{Notation}

Any Gaussian integer $q$ can be written uniquely in the form
\begin{equation}\label{eqn:gint}
  \centering
  q=u \cdot p_1^{\alpha_1}p_2^{\alpha_2}\ldots p_k^{\alpha_k}
\end{equation}
where $u \in \{ \pm 1, \pm i \}$, $\alpha_i \geq 1$, and the $p_i$ are Gaussian primes such that $p_i \neq p_j$ when $i \neq j$, $\RE{p_i} > 0$ and $\IM{p_i} \geq 0$.  We denote the set of such $q$ for which $\RE{q}>0$ and $\IM{q}\geq 0$ by $\ZZ[i]^+$.  In the following, for all $q \in \ZZ[i]$, $\sum_{d \mid q}$ denotes a sum over $d\in \ZZp$ which divide $q$.  We define complex Mobius and Euler-phi functions as follows.
\begin{defn} For a Gausian integer $q$ as in \eqref{eqn:gint}, define $\mu_i : \ZZp \rightarrow \ZZ$ by
\begin{equation*}
  \mu_i(q) :=
    \left\{
	\begin{array}{ll}
        1 & \mbox{if } q=u, \\
		(-1)^k  & \mbox{if } \alpha_1 = \cdots = \alpha_k = 1, \\
		0 & \mbox{otherwise. }
	\end{array}
\right.
\end{equation*}
\end{defn}

\begin{defn} For a Gaussian integer $q$ as in \eqref{eqn:gint}, define $\phi_i : \ZZp \rightarrow \ZZ$ by
\begin{equation*}
  \phi_i(q) := \left\lvert \left( \ZZ[i] / q\ZZ[i] \right)^*\right\rvert .
\end{equation*}
\end{defn}

For a function $f:\ZZp \rightarrow \CC$, we also make the following definition,

\begin{defn}
$f$ is called multiplicative if
\begin{equation}\label{eq:mult}
f(qr) = f(q)f(r)
\end{equation}
is true whenever $(q,r) = 1$.  $f$ is called completely multiplicative if \eqref{eq:mult} holds for all $q, r \in \ZZp$.
\end{defn}

\subsection{Elementary Lemmas}

\begin{lem} For $q$ as in \eqref{eqn:gint} we have
\begin{equation*} \sum_{d|q} \mu_i(d) =
  \ifotherdef{1}{q=u,}{0}{}
  \end{equation*}
\end{lem}

\begin{proof}
Clearly this is true when $q=u$ so suppose $q \neq u$.  Then we have
\begin{align*}
        \sum\limits_{d|q} \mu_i(d) &= \sum\limits_{d \mid p_1 \ldots p_k} \mu_i(d) \\[2ex]
          &= 1 - \sum\limits_{p_i} 1 + \sum\limits_{\substack{p_i,p_j \\ i\neq j}} 1 - \cdots \\[2ex]          &= (1-1)^k \\
          &= 0.
\end{align*}
\end{proof}

\begin{lem}
Two functions $f,g : \ZZp \rightarrow \CC$ satisfy
\begin{equation}\label{eq:mobiussum1}
  \centering
        f(q) = \sum_{d \mid q} g(d) \text{,    }
\end{equation}
for all $q \in \ZZp$ if and only if they satisfy
\begin{align}\label{eq:mobiussum2}
        g(q) &= \sum\limits_{d \mid q} \mu_i(d)f(\frac{q}{d}) \nonumber \\
        &= \sum\limits_{d \mid q} \mu_i(\frac{q}{d})f(d) \text{,    }
\end{align}
for all $q \in \ZZp$.
\end{lem}

\begin{proof}Suppose (\ref{eq:mobiussum1}) holds.  Then
\begin{align*}
  \sum\limits_{d \mid q} \mu_i(\frac{q}{d})f(d) &= \sum\limits_{d \mid q} \mu_i(\frac{q}{d})\sum\limits_{e\mid d} g(e) \\[2ex]   &= \sum\limits_{e \mid q} g(e) \sum\limits_{d'\mid \frac{q}{e}} \mu_i(\frac{q}{d'e}) \\[2ex]   &= \sum\limits_{e \mid q} g(e) \sum\limits_{d'\mid \frac{q}{e}} \mu_i(d') \\[2ex]
   &= g(q)
\end{align*}
since $\sum\limits_{d'\mid \frac{q}{e}} \mu_i(d') = 0$ unless $\frac{q}{e} = 1$.  The other direction is proved similarly.
\end{proof}

\begin{lem}
For every $q \in \ZZp$ we have
\begin{equation*}
  \centering
  \abs{q}^2 = \sum\limits_{d \mid q} \phi_i(d)
\end{equation*}
\end{lem}

\begin{proof}
  Let $U = \{x+iy | 0 \leq x,y <1 \}$.  Then $U_q := q \cdot U$ is a fundamental domain for $\nicefrac{\CC}{q\ZZ[i]}$ and we have
  \begin{align*}
    \lvert \nicefrac{\ZZ[i]}{q\ZZ[i]} \rvert &= \lvert U_q \cap \ZZ[i] \rvert \\[2ex]        &= \sum\limits_{d \mid q} \#\{r \in U_q \cap \ZZ[i] \mid gcd(r,q) = d\} \\
        &= \sum\limits_{d \mid q} \#\{r \in U_{\nicefrac{q}{d}} \cap \ZZ[i] \mid gcd(r,\frac{q}{d}) = 1\} \\
        &= \sum\limits_{d \mid q} \phi_i(\frac{q}{d})   \ = \ \sum\limits_{d \mid q} \phi_i(d) \, .
  \end{align*}
  Also, we have $\lvert \nicefrac{\ZZ[i]}{q\ZZ[i]} \rvert = vol(U_q) = \abs{q}^2$, so we are done.
\end{proof}
The next result follows directly from the previous two Lemmas.
\begin{lem} For every $q$ in $\ZZp$ we have
\begin{equation*}
  \centering
  \phi_i(q) = \abs{q}^2 \sum\limits_{d \mid q} \frac{\mu_i(d)}{\abs{d}^2}.
\end{equation*}
\label{lem:phii}
\end{lem}
For the next Lemma we need to define the sum of squares function $r_2(n)$.
\begin{defn}
For a positive integer $n$, define $r_2(n)$ by
\begin{equation*}
  r_2(n) = \#\{(a,b) \in \ZZ^2 \mid a^2+b^2 = n\}.
\end{equation*}
\end{defn}

\begin{lem} \label{lem:sumr2}
For $a\geq 0$,
\begin{equation*}
  \sum\limits_{n \leq N} n^a r_2(n) = \frac{\pi N^{a+1}}{a+1} + \mathcal{O}(N^{a+\kappa})
\end{equation*}
with $\frac{1}{4} < \kappa < \frac{1}{2}$.
\end{lem}

\begin{proof}Using partial summation and the fact from \cite{GCP} that $\sum\limits_{n\leq N} r_2(n)=\pi N + \bgoh{N^{\kappa}}$ for $\frac{1}{4}<\kappa<\frac{1}{2}$, we have
\begin{align*}
    \sum\limits_{n \leq N} n^a r_2(n) &= \sum\limits_{n \leq N} (n^a - (n+1)^a)\sum\limits_{k \leq n} r_2(k) + (N+1)^a \sum\limits_{n \leq N} r_2(n) \\[2ex]     &= \sum\limits_{n \leq N} (n^a - (n+1)^a)(\pi n + \mathcal{O}(n^{\kappa})) + (N+1)^a (\pi N + \mathcal{O}(N^{\kappa})) \\[2ex]     &= \pi \sum\limits_{n \leq N} n^a + \mathcal{O}(\sum\limits_{n \leq N} n^{a-1+\kappa}) + \mathcal{O}(N^{a+\kappa}) \\[2ex]     &= \frac{\pi}{a+1} N^{a+1} + \mathcal{O}(N^{a+\kappa}).
\end{align*}
\end{proof}

\begin{lem} \label{lem:sumq>Q}
For $s>1$, we have
\begin{equation*}
    \sum\limits_{\abs{q} \geq Q} \frac{1}{\abs{q}^{2s}} \ll_s \frac{1}{Q^{2(s-1)}}.
\end{equation*}
\end{lem}

\begin{proof} Rewriting the left hand side as a sum over annuli and using the fact that $\sum\limits_{n\leq N} r_2(n)=\pi N + \bgoh{N^{\kappa}}$, we have
\begin{align*}
    \sum\limits_{\abs{q} \geq Q} \dfrac{1}{\abs{q}^{2s}} &= \sum\limits_{k=0}^{\infty} \sum\limits_{2^kQ \leq \abs{q} <2^{k+1}Q} \dfrac{1}{\abs{q}^{2s}} \\[2ex]    &\leq  \sum\limits_{k=0}^{\infty} \dfrac{1}{2^{2ks}Q^{2s}} \left( \sum\limits_{n \leq 2^{2\abs{k+1}}Q^2}r_2(n) - \sum\limits_{n \leq 2^{2k}Q^2} r_2(n) \right) \\[2ex]    &\ll \sum\limits_{k=0}^{\infty} \dfrac{2^{2k}Q^{2}}{2^{2ks}Q^{2s}}\\[2ex]    &\ll \dfrac{1}{Q^{2(s-1)}} \sum\limits_{k=0}^{\infty} \dfrac{1}{2^{2k(s-1)}} \\[2ex]    &\ll_s  \dfrac{1}{Q^{2(s-1)}}.
\end{align*}
\end{proof}

\begin{fact}
We have, for $Re(s)>1$ and Gaussian primes $p$,
\begin{align*}
    \zeta_i(s) &:= \sum\limits_{q \in \ZZp} \dfrac{1}{\abs{q}^{2s}}\\
    &= \prod\limits_{\substack{p \in \ZZp \\ (p \text{ prime})}} (1-\abs{p}^{-2s})^{-1},
\end{align*}
and
\begin{equation*}
  \zeta_i^{-1}(s) = \sum\limits_{q \in \ZZp} \frac{\mu_i(q)}{\lvert q \rvert^{2s}}.
\end{equation*}
\end{fact}

\begin{lem}\label{lem:sumphii}
%Lemma 4.6 on summming phi_i(q)
For a positive integer $Q$, we have
\begin{equation*}
  \sum\limits_{\substack{q\in \ZZp \\ \abs{q} \leq Q}} \phi_i(q) = \frac{\pi}{8} \zeta_i^{-1} \left(2 \right)Q^4 + \mathcal{O}\left(Q^{2+2\kappa}\right).
\end{equation*}
\end{lem}

\begin{proof}By Lemmas \ref{lem:phii} and \ref{lem:sumr2}, we have
  \begin{align*}
    \sum\limits_{\substack{q\in \ZZp \\ \abs{q} \leq Q}} \phi_i(q) &= \sum\limits_{\substack{q\in \ZZp \\ \abs{q} \leq Q}} \abs{q}^2 \sum\limits_{d \mid q} \dfrac{\mu_i(d)}{\abs{d}} \\[2ex]
    &= \sum\limits_{\substack{d\in \ZZp \\ \abs{d} \leq Q}} \mu_i(d) \sum\limits_{\substack{q'\in \ZZp \\ \abs{q'} \leq \frac{Q}{\abs{d}}}} \abs{q'}^2  \\[2ex]
    &= \sum\limits_{\substack{d\in \ZZp \\ \abs{d} \leq Q}} \mu_i(d) \sum\limits_{k \leq \frac{Q^2}{\abs{d}^2}}  \dfrac{r_2(k)}{4} k  \\[2ex]
    &= \frac{1}{4} \sum\limits_{\substack{d\in \ZZp \\ \abs{d} \leq Q}} \mu_i(d) \left(\dfrac{\pi Q^4}{2\abs{d}^4} + \mathcal{O}\left(\dfrac{Q^{2+2\kappa}}{\abs{d}^{2+2\kappa}}\right)\right) \\[2ex]
    &= \frac{\pi}{8} \zeta_i^{-1}\left(2\right)Q^4 + \mathcal{O}\left(Q^{2+2\kappa}\right),
  \end{align*}
using Lemma \ref{lem:sumq>Q} and the fact above with $s=2$ and $s=1+\kappa$.
\end{proof}

Now, for a multiplicative function $f:\ZZp \rightarrow \CC$ (with respect to \eqref{eqn:gint}), we have the following easily proved result.
\begin{lem} \label{lem:multiplicative} If $f$ is multiplicative then,
\begin{equation*}
  \sum\limits_{d \mid q} f(d) = \prod\limits_{i=1}^{k} (f(1) + f(p_i) + \ldots + f(p_i^{\alpha_i})).
\end{equation*}
If $f$ is completely multiplicative then
\begin{equation*}
  \sum\limits_{d \mid q} f(d) = \prod\limits_{i=1}^k \left(\frac{f(p_i)^{\alpha_i + 1}-1}{f(p_i) - 1} \right).
\end{equation*}
\end{lem}

\begin{lem} \label{lem:abel} \textbf{\emph{Abel's Summation Formula}}  Suppose we have functions $a:\NN \rightarrow \RR$ and $f:\RR \rightarrow \RR$, and that $f'(x)$ exists and is continuous.  Let $A(x)=\sum\limits_{n \leq x} a(n)$.  Then
\begin{equation*}
  \centering
  \sum\limits_{n \leq x}a(n)f(n) = A(x)f(x) - \int_{1}^{x} \! A(t)f'(t) \, \mathrm{d}t.
\end{equation*}
\end{lem}

\section{Preliminary Lemmas - Ford Spheres}

Before we can begin looking at moments for Ford spheres we will need to gather a few basic results about their structure.  In Section \ref{sec:gencff} we introduce a complex version of mediants and show that this can be used to generate every Gaussian rational in $I_2$ if we start with $0, 1, i$ and $1+i$.  In Section \ref{sec:consecspheres} we classify when two Gaussian integers appear as denominators of consecutive fractions.  Finally, in Section \ref{sec:countdenoms} we consider, for a given Gaussian integers $s$, how to count the Gaussian integers $s'$ which are the denominator of a fraction that is consecutive to a fraction of the form $\frac{r}{s}$.

\subsection{Generating $\GS$} \label{sec:gencff}

The Farey fractions are generated from 0 and 1 by taking mediants, and we have seen that doing this will eventually generate every rational number in $(0,1)$.  In the complex case, we must begin with 0, 1, $i$ and $1+i$ and take a complex version of mediants instead.  Given two adjacent fractions $\frac{r}{s}$ and $\frac{r'}{s'}$, the four fractions which are adjacent to both are given by
  \begin{equation}\label{eqn:compmeds}
  \centering
    \dfrac{r+ur'}{s+us'} \quad \text{ for } \quad u \in \{ \pm 1, \pm i\}.
  \end{equation}
Unlike taking mediants of the usual Farey fractions, which always results in a denominator larger than that of either initial fraction, this type of mediant can result in a smaller denominator.  Analogously, when taking mediants of real Farey fractions $\frac{p}{q}$ and $\frac{p'}{q'}$ we could instead consider
  \begin{equation*}
  \centering
    \dfrac{p+u'p'}{q+u'q'} \text{  for } u' \in \{ \pm 1 \}.
  \end{equation*}
This would sometimes give us a denominator smaller than either $q$ or $q'$, specifically when we take $u'=-1$.  Taking mediants in this way only gives us new Farey fractions when $u'=1$, so we can discard $u'=-1$ and still generate every fraction.  However, when taking complex mediants it is not clear which choices of unit $u$ will lead to larger denominators and which will lead to smaller, so we need to consider all four units and ignore any repeated resulting fractions.

We now prove that beginning with 0, 1, $i$, and $1+i$, taking mediants as in (\ref{eqn:compmeds}) will generate every Gaussian rational in the unit square of $\mathbb{C}$.
\begin{lem}\label{lem:EveryCFF}
Given Gaussian integers $r=r_1+ir_2$ and $s=s_1+is_2$ such that $\frac{r}{s} \in I_2$ and $(r,s) = u$ for some unit $u$ in $\ZZ[i]$, $\frac{r}{s}$ occurs as a complex mediant of two consecutive fractions in $I_2$ with denominators of modulus less than $\abs{s}$.
\end{lem}

\begin{proof}
We argue by induction on $\abs{s}$.
Assume that all Gaussian rationals in $I_2$ with denominator of modulus less than $\abs{s}$ have already been found.
Now, since $gcd(r,s) = u$ we can find $x,y \in \ZZ[i]$ such that
\begin{equation*}
  \centering
  rx-sy = u
\end{equation*}
with $\abs{x} < \abs{s}$ and $\frac{y}{x} \in I_2$.  So by the assumption, $\frac{y}{x}$ has already been found.
Let
\begin{equation*}
  \frac{a}{b} = \frac{r-vy}{s-vx}
\end{equation*}
where $v$ is a unit to be decided later.  We claim that $\frac{a}{b}$ is adjacent to $\frac{r}{s}$ and $\frac{y}{x}$, and $\abs{b}<\abs{s}$.  First,
\begin{align*}
    ax-by &= (r-vy)x-(s-vx)y\\
        &= rx - sy\\
        &= u
\end{align*}
so $\frac{a}{b}$ is adjacent to $\frac{y}{x}$.  Similarly,
\begin{align*}
    as-br &= (r-vy)s-(s-vx)r \\
        &= v(rx - sy) \\
        &= vu \in \ZZ[i]^*
\end{align*}
so $\frac{a}{b}$ is adjacent to $\frac{r}{s}$.

\begin{figure}
    \centering
    \includegraphics[width=8cm]{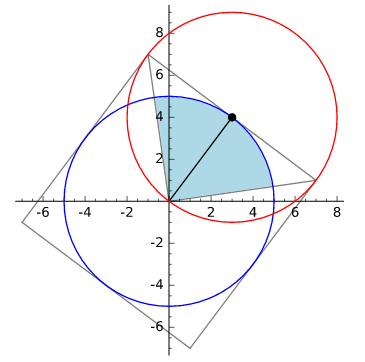}
    \caption{Diagram for the proof of Lemma \ref{lem:EveryCFF}.  The point $s$ is shown in black.  UC is the red circle, LC is the blue circle and QC is the shaded quarter circle.}
    \label{fig:EveryCFF}
\end{figure}

Now, to show that $\abs{b}<\abs{s}$, consider Figure \ref{fig:EveryCFF}.  The circle $UC$ contains all points which are within $\abs{s}$ of $s$, this is where we want $vx$ to be.  The circle $LC$ contains all the possible locations of $x$, since $\abs{x} < \abs{s}$.  If we split $LC$ into quarters we can force $vx$ to lie in any one of those quarters by choosing the unit $v$ accordingly.  In particular, we can choose $v$ so that $vx$ lies in $QC$, and so in $UC$. Thus $\abs{b} = \abs{s-vx} < \abs{s}$ as required.

Now all that is left to check is that $\frac{a}{b} \in I_2$.  We know that $\frac{r}{s}, \frac{y}{x} \in I_2$ so there are three possibilities:
\begin{enumerate}
  \item $\frac{r}{s}, \frac{y}{x} \in \mathring{I}_2$,
  \item $\frac{r}{s}, \frac{y}{x} \in I_2 \backslash \mathring{I}_2$, or
  \item $\frac{r}{s} \in \mathring{I}_2$ and $\frac{y}{x} \in I_2 \backslash \mathring{I}_2$ or vice versa.
\end{enumerate}
In cases $1$ and $3$ at least one of the two fractions is in $\mathring{I}_2$ and $\frac{a}{b}$ is adjacent to that fraction so, as their spheres are tangent, we must have $\frac{a}{b} \in I_2$.  In case 2 however, it is possible to choose $\frac{a'}{b'} \notin I_2$ with $\abs{b'} < \abs{s}$ which is adjacent to both $\frac{r}{s}$ and $\frac{y}{x}$.  But in this case we could always instead choose $\frac{a}{b} \in I_2$, the mirror image of $\frac{a'}{b'}$ over the boundary of $I_2$.  Thus, $\frac{a}{b} \in I_2$.
Note also that the fractions $\frac{a}{b}$ and $\frac{y}{x}$ are consecutive in $\GS$ for $S$ such that $S < \abs{s}$ as their spheres are tangent (since they are adjacent) and $\frac{r}{s}$ is a fraction with a sphere which is tangent to both and has radius less than $\frac{1}{2S^2}$.
\end{proof}

\subsection{Classifying Consecutive for Ford Spheres} \label{sec:consecspheres}
In Section \ref{sec:spheres} we give a geometric definition of consecutivity for Ford spheres.  However, for our moment calculation we will need a set of criteria that tell us exactly when a given pair of denominators are consecutive in $\GS$.

In $\FQ$, $q$ and $q'$ are called consecutive if they are denominators of two fractions which are consecutive.  For $\FQ$ we have the following classification of consecutivity for denominators.

\begin{lem} Denominators $q$ and $q'$ will be consecutive in $\FQ$ if and only if all of the following are satisfied:
\begin{enumerate}
  \item $1 \leq q, q' \leq Q$,
  \item $(q,q') = 1$, and
  \item $q+q' > Q$.
\end{enumerate}
Furthermore, for each pair of denominators $q, q'$ satisfying these conditions there will be exactly two pairs of consecutive fractions with denominators $q$ and $q'$.  In one case $\frac{p_1}{q} < \frac{p_1'}{q'}$, and in the other $\frac{p_2}{q}> \frac{p_2'}{q'}$
\end{lem}

In $\GS$, $s$ and $s'$ are called consecutive if they are denominators of two fractions which are consecutive.  Now that we know how to generate Farey fractions in $\mathbb{C}$, we can also classify what it means for two denominators to be consecutive in $\GS$.  In the case of the usual Farey fractions, the three requirements for $q, q' \in \ZZ$ to be consecutive can be thought of as
\begin{enumerate}
  \item The fractions are in $\FQ$. ($q, q' \leq Q$)
  \item The fractions are adjacent.  ($(q,q') = 1$)
  \item There is no smaller fraction between the original two which is also in $\FQ$.  ($q+q' > Q$)
\end{enumerate}
The classification for $\mathbb{C}$ should have conditions analogous to these statements, but take into account that we are now taking complex mediants.  Such a classification is given below.

\begin{lem}\label{lem:consec} Two Gaussian integers $s$ and $s'$ appear as consecutive denominators in $\GS$ if and only if all of the following are satisfied:
\begin{enumerate}
  \item $\abs{s}, \abs{s'} \leq S$,
  \item $(s,s') = 1$, and
  \item $\abs{s'+u's} > S$ for some unit $u'$.
\end{enumerate}
Furthermore, there are exactly four distinct pairs $r,r' \in \ZZ[i]$ which give consecutive fractions $\frac{r}{s}, \frac{r'}{s'}$, when these three conditions are satisfied.
\end{lem}

The three conditions follow directly from the geometric definition of consecutive in $\GS$, which we previously described.  The final statement is proved as follows.

\begin{proof}
Suppose $s,s'\in \ZZ[i]$ satisfy the three conditions above.  Then there are $r,r' \in \ZZ[i]$ for which $\frac{r}{s}$ and $\frac{r'}{s'}$ are consecutive in $\GS$.  So for $r,r'$ we have
\begin{equation*}
    rs'-r's = u
\end{equation*}
where $u$ is a unit in $\ZZ[i]$.  There are four choices for the unit u, each of which corresponds to a pair $r,r'$. We claim that each of these four pairs is distinct.  We have
\begin{align*}
    r &= us'^{-1} \mod{s}, and \\
    r' &= \frac{rs' - u}{s},
\end{align*}
so $r'$ is determined by $r$.  When the unit $u$ is changed either $r$ is changed, or $r$ remains the same and so $r'$ is changed.  Either way a new pair is found for each choice of $u$, and so there are four distinct possibilities for the pair $r,r'$.
\end{proof}

\subsection{Counting Consecutive Denominators} \label{sec:countdenoms}

To estimate $\mathcal{M}_{k,I_2}(S)$ it will be necessary, given $s\in \ZZp$ with $\abs{s} \leq S$, to count how many different $s'\in \ZZp$ have at least one fraction $\frac{r'}{s'} \in \GS$ which is consecutive to a fraction with denominator $s$.  In other words, given $s$, how many $s'\in \ZZp$ satisfy the three conditions of Lemma \ref{lem:consec}?

Ignoring the coprimality condition for now, we need to know how many $s'$ satisfy
\begin{enumerate}
  \item $|s'| \leq S$, and
  \item $|s'+us| > S$ for some unit $u$,
\end{enumerate}
for a given $s$.
The $s'$ satisfying condition 1 are those points on the $\ZZ[i]$ lattice that lie inside $R$, the circle of radius $S$ centred on the origin.  In condition 2 we consider mediants of $s'$ with $s$, taking $s'+us$ for each unit $u \in \{\pm 1, \pm i\}$.  For $s'$ to satisfy condition 2, one of these four points must lie outside of $R$.  We can look at this condition in another way by translating $R$.  For example, consider $s'$ for which $\abs{s+s'} > S$, so $s+s'$ lies outside of $R$.  Then if we translate $R$ by $-s$, the point $s'$ will lie outside of the translated circle.  Similarly, if $s'$ has $\abs{s'+us}>S$, $s'+us$ lies outside of $R$ and so $s'$ lies outside the circle of radius $S$ centred at $-us$.

Translating the circle $R$ in each of the four directions $s$, $-s$, $is$ and $-is$, we have the picture in Figure \ref{fig:regions}.  Points $s'$ satisfy the two conditions above iff they lie inside the red circle $R$ and outside at least one of the blue circles, i.e. in the shaded area.  Our aim then is to count the points on the $\ZZ[i]$ lattice in this region that are coprime to $s$.

\begin{figure}
    \centering
    \includegraphics[width=9cm]{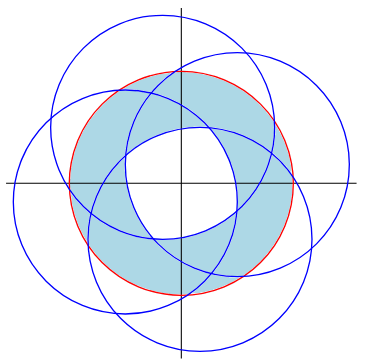}
    \caption{Circles of radius $S$ with centres the origin, $\pm s$ and $\pm is$.  To be consecutive to $s$, denominators $s'$ must lie in the shaded region.}
    \label{fig:regions}
\end{figure}

We will denote the shaded region in this diagram by $\Omega$ and its boundary by $\partial \Omega$.  The following theorem concerns any region $\Omega$ in the complex plane, but holds for our region $\Omega$ from Figure \ref{fig:regions} in particular.

\begin{thm}For a region $\Omega$ in the complex plane with boundary $\partial \Omega$, we have
\begin{equation*}
\centering
  \sum\limits_{\substack{z \in \Omega \\ (z,s)=1}} 1 = \dfrac{\phi_i(s)}{\abs{s}^2} \abs{\Omega} + \mathcal{O}_{\epsilon} \left( \abs{\partial \Omega} \abs{s}^{\epsilon} \right)
\end{equation*}
for all $\epsilon > 0$.
\label{thm:latticepointcount}
\end{thm}

\begin{proof} By Lemma \ref{lem:phii}, we have
\begin{align*}
  \sum\limits_{\substack{z \in \Omega \\ (z,s)=1}} 1 &= \sum\limits_{z \in \Omega} \sum\limits_{d | z,s} \mu_i(d) \\
  &= \sum\limits_{d|s} \mu_i(d) \sum\limits_{\substack{z \in \Omega \\ d|z }} 1 \\
  &= \sum\limits_{d|s} \mu_i(d) \left( \dfrac{\abs{\Omega}}{\abs{d}^2} + \mathcal{O}\left(\dfrac{\abs{\partial \Omega}}{\abs{d}} \right)\right) \\
  &= \dfrac{\phi_i(s)}{\abs{s}^2}\abs{\Omega} + \mathcal{O}\left(\abs{\partial \Omega} \sum\limits_{d|s} \dfrac{\abs{\mu_i(d)}}{\abs{d}} \right).
\end{align*}

Now, using Lemma \ref{lem:multiplicative}, we have that
\begin{equation*}
 \sum\limits_{d|s} \frac{\abs{\mu_i(d)}}{\abs{d}} = \prod_{p|s} \left( 1 + \frac{1}{\abs{p}}\right),
\end{equation*}
and we also observe that
\begin{equation*}
  \log\left(\prod_{p|s} \left( 1 + \frac{1}{\abs{p}}\right)\right) \ll \sum\limits_{p|s} \frac{1}{\abs{p}},
\end{equation*}
where $p$ is a Gaussian prime.  Therefore, the worst case scenario is when $s$ is the product of the smallest possible distinct primes (replacing any such prime with a larger prime will reduce the value of the sum).  So we have
\begin{equation*}
  \centering
  \sum\limits_{p|s} \frac{1}{\abs{p}} \leq \sum\limits_{\abs{p}^2 \leq X} \frac{1}{\abs{p}}
\end{equation*}
for some $X$ depending on $s$.

To estimate $X$, first note that
\begin{align*}
  \abs{s} &= \prod\limits_{\abs{p}^2 \leq X} \abs{p} \text{, and} \\
  \log{\prod\limits_{\abs{p}^2 \leq X} \abs{p}} &= \sum\limits_{\abs{p}^2 \leq X} \log{\abs{p}}.
\end{align*}
Now, using Stieltjes integration and the Prime Number Theorem for Gaussian primes (Proposition 7.17 in \cite{PNT}),
\begin{align*}
  \sum\limits_{\abs{p}^2 \leq X} \log{\abs{p}} &= \displaystyle \int_{1^-}^{X^+} \! (\log{t^{\frac{1}{2}}}) \, \mathrm{d}\pi_i(t) \\
  &= \dfrac{(\log{X})\pi_i(X)}{2} - \displaystyle \int_{1}^{X} \! \dfrac{\pi_i(t)}{2t} \, \mathrm{d}t \\
  &= \dfrac{X}{2} + o(X),
\end{align*}
where $\pi_i(t)$ is the prime counting function for Gaussian integers, which counts Gaussian primes $p$ with $\abs{p}^2 \leq t$.  So
\begin{equation*}
\centering
    \prod\limits_{\abs{p}^2 \leq X} \abs{p} = e^{\frac{X}{2} + o(X)}.
\end{equation*}
If $X \gg \log{\abs{s}}$, this says that $\abs{s} = \prod\limits_{\abs{p}^2 \leq X} \abs{p} \gg \abs{s}$, so we must have $X \leq c\log{\abs{s}}$, for some $c>0$.
So using Stieltjes integration and the Prime Number Theorem for Gaussian primes again, we have
\begin{align*}
  \sum\limits_{p|s} \dfrac{1}{\abs{p}} & \leq \sum\limits_{\abs{p}^2 \leq c\log{\abs{s}}} \dfrac{1}{\abs{p}} \\
  &= \displaystyle \int_{1^-}^{(c\log{\abs{s}})^+} \! t^{-\frac{1}{2}} \, \mathrm{d}\pi_i(t) \\
  &= \dfrac{\pi_i(c\log{\abs{s}})}{c^{\frac{1}{2}}(\log{\abs{s}})^{\frac{1}{2}}} + \dfrac{1}{2} \displaystyle \int_{1}^{c\log{\abs{s}}} \! \dfrac{\pi_i(t)}{t^{\frac{3}{2}}} \, \mathrm{d}t \\
  &\ll \dfrac{(\log{\abs{s}})^{\frac{1}{2}}}{\log({\log{\abs{s}}})}
\end{align*}
Thus, we have
\begin{align*}
  \sum\limits_{d|s} \dfrac{\abs{\mu_i(d)}}{\abs{d}} &= \exp\left(\mathcal{O}\left(\sum\limits_{p|s} \dfrac{1}{\abs{p}} \right)\right) \\
  &= \exp\left(\mathcal{O}\left(\dfrac{(\log{\abs{s}})^{\frac{1}{2}}}{\log{(\log{\abs{s}})}}\right)\right) \\
  &\ll_{\epsilon} \abs{s}^{\epsilon}
\end{align*}
for all $\epsilon > 0$.
\end{proof}

We now show that in this estimation the main term will always be asymptotically larger than the error term when $\Omega$ is the shaded region region in Figure \ref{fig:regions}.  In this case the sum in Theorem \ref{thm:latticepointcount} is counting points $s'$ which satisfy all three conditions for being consecutive to $s$ and so is equal to the inner sum in \eqref{eqn:firstmoment1}.  The following argument uses only the area of the region and the length of its boundary, not its position.  So to simplify the calculations we rotate our view of the diagram so that the circles' centres lie on the axes, as shown in Figure \ref{fig:rotated}.

\begin{figure}
    \centering
    \includegraphics[width=8cm]{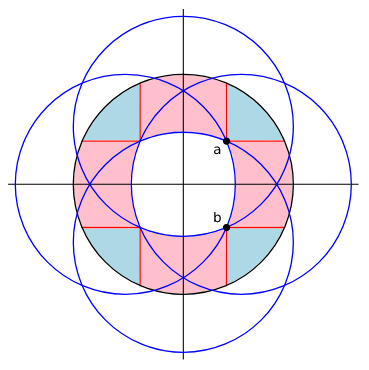}
    \caption{The rotated view of figure \ref{fig:regions}.  The shaded area is $\Omega$. }
    \label{fig:rotated}
\end{figure}

We call the top right blue corner region $C$ and the right hand red region $A$.  Clearly, $Area(\Omega) = 4(Area(A) + Area(C))$ and $\abs{\partial \Omega} \ll S$.
Let $A_h$ be the height of $A$.  This will be given by the difference in the imaginary parts of the points $a$ and $b$.  The symmetry of the diagram means that $Im(b)=-Im(a)$ and so their difference is $2Im(a)$.

Now, the circles with intersection point $a$ have equations
\begin{align*}
 y^2+(x+\abs{s})^2 & = S^2 \text{ , and}  \\
    (y+\abs{s})^2 +x^2 & = S^2.
\end{align*}
The points of intersection of these two circles lie on the line $y=x$.  Substituting this into one of the equations gives us
\begin{equation*}
  \centering
  2y^2+2\abs{s}y+\abs{s}^2-S^2 = 0.
\end{equation*}
The point $a$ has positive imaginary part and so $Im(a)$ is the positive solution to this equation.
\begin{equation*}
    Im(a) =\frac{1}{2} \left(-\abs{s}+\sqrt{2S^2-\abs{s}^2}\right)
\end{equation*}
Thus, the height of $A$ is
\begin{align*}
    A_h &= 2Im(a) \\
        &= \sqrt{2S^2-\abs{s}^2}-\abs{s}\\
        &= \sqrt{2}S\left(\sqrt{1-\frac{\abs{s}^2}{2S^2}}-\frac{\abs{s}}{\sqrt{2}S}\right),
\end{align*}
and
\begin{align*}
Area(A)&= \abs{s} A_h \\
&= \sqrt{2}\abs{s}S\left(\sqrt{1-\frac{\abs{s}^2}{2S^2}}-\frac{\abs{s}}{\sqrt{2}S}\right).
\end{align*}
For the area of $C$ note that the two red lines at the edges of $C$ are two sides of a square of side length $\abs{s}$ which completely contains $C$.  Further, $C$ will always make up more than half of this square and so,
\begin{equation*}
  Area(C) \geq \frac{1}{2}\abs{s}^2.
\end{equation*}
Now, if $S \leq 2\abs{s}$, note that $Area(C) \geq \frac{\abs{s}^2}{2}$ so $Area(\Omega) \geq 2\abs{s}^2$, and $\abs{\partial \Omega}\abs{s}^{\epsilon} \ll \abs{s}^{1+\epsilon}$.  On the other hand, if $S > 2\abs{s}$, then $Area(\Omega) \geq 4Area(A) \geq  2(\sqrt{7}-1)S\abs{s}$, and $\abs{\partial \Omega}\abs{s}^{\epsilon} \ll S\abs{s}^{\epsilon}$. So for any choice of $S$ and $\abs{s}$, the error term is (on average over s) asymptotically smaller than the main term in Theorem \ref{thm:latticepointcount}.

\section{First Moment}

In this section we aim to prove Theorem \ref{thm:mainthm}.  In place of $\mathcal{M}_{k,I}(Q)$, we are looking at $k$-moments
\begin{equation*}
  M_{k,I_2}(S) = \sum\limits_{\substack{\frac{r}{s},\frac{r'}{s'}\in \GS \\ \text{consec}}} \left( \dfrac{1}{2\abs{s}^2} + \dfrac{1}{2\abs{s'}^2}\right)^k,
\end{equation*}
so the first moment for Ford spheres is
\begin{equation*}
  M_{1,I_2}(S) = \sum\limits_{\substack{\frac{r}{s},\frac{r'}{s'}\in \GS \\ \text{consec}}} \left( \dfrac{1}{2\abs{s}^2} + \dfrac{1}{2\abs{s'}^2}\right).
\end{equation*}
Due to Lemma \ref{lem:consec}, this can be rewritten as
\begin{equation*}
  M_{1,I_2}(S) = \dfrac{1}{2} \sum\limits_{\substack{s\in \ZZp\\ \abs{s} \leq S}} \sum\limits_{\substack{s'\in \ZZp\\ s' \text{ consec to } s }} 4 \left( \dfrac{1}{2\abs{s}^2} + \dfrac{1}{2\abs{s'}^2}\right)
\end{equation*}
where the factor of 4 comes from the lemma and the factor of $\frac{1}{2}$ means we are not counting both ``$s$ is consecutive to $s'$ '' and ``$s'$ is consecutive to $s$''.  Now, $s$ and $s'$ run though the same numbers so, for every $a\in \ZZp$, every time $s=a$ produces a term $\frac{1}{\abs{a}^2}$, $s'=a$ will produce another $\frac{1}{\abs{a}^2}$ term.  So, using Theorem \ref{thm:latticepointcount}, we can write
\begin{align}    \label{eqn:firstmoment1}
    M_{1,I_2}(S) &= 2 \sum\limits_{\substack{s\in \ZZp\\ \abs{s} \leq S}} \sum\limits_{\substack{s'\in \ZZp\\ s' \text{ consec to } s }} \dfrac{1}{\abs{s}^2} \nonumber \\
    &= 2 \sum\limits_{\substack{s\in \ZZp\\ \abs{s} \leq S}} \dfrac{1}{\abs{s}^2} \sum\limits_{\substack{s'\in \ZZp\\ s' \text{ consec to } s }} 1 \\
    &= 2 \sum\limits_{\substack{s\in \ZZp\\ \abs{s} \leq S}} \dfrac{1}{\abs{s}^2}\sum\limits_{\substack{z \in \Omega \\ (z,s)=1}} 1 \nonumber \\
    &= 2 \sum\limits_{\substack{s\in \ZZp\\ \abs{s} \leq S}} \dfrac{1}{\abs{s}^2} \left( \dfrac{\phi_i(s)}{\abs{s}^2} \abs{\Omega} + \mathcal{O}_{\epsilon} \left( \abs{\partial \Omega} \abs{s}^{\epsilon} \right) \right) \nonumber \\
    &= 2 \sum\limits_{\substack{s\in \ZZp\\ \abs{s} \leq S}} \dfrac{\phi_i(s)}{\abs{s}^4}\abs{\Omega} + \mathcal{O}_{\epsilon} \left(\sum\limits_{\substack{s\in \ZZp\\ \abs{s} \leq S}} \dfrac{1}{\abs{s}^{2-\epsilon}} \abs{\partial \Omega} \right) \nonumber \\
    &=: 2A + \bigoh_{\epsilon}(B),
    \label{eqn:firstmoment2}
\end{align}
for all $\epsilon > 0$, where $\Omega$ is the shaded region in figure \ref{fig:regions} and its boundary is $\partial \Omega$.  We now aim to estimate $A$ and $B$.

\subsection*{The Area of $\Omega$}

In order to estimate $A$ we need to know the area of our region, which will be calculated using polar coordinates.

\begin{prop} \label{prop:areaomega}
The area of the region $\Omega$ is given by
\begin{equation*}
\centering
  \abs{\Omega} = -2\abs{s}^2 + I_1
\end{equation*}
where $I_1=8S^2 \int_0^{\sin^{-1}\left(\frac{\abs{s}}{\sqrt{2}S}\right)} \! \cos^2{u} \, \mathrm{d}u$.
\end{prop}

\begin{proof}

\begin{figure}
    \centering
    \includegraphics[width=8cm]{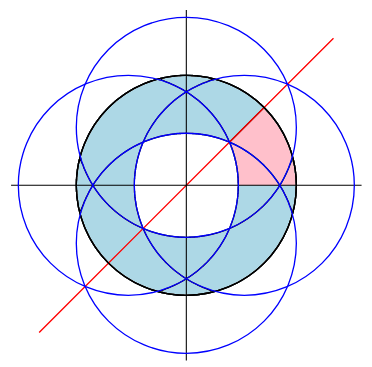}
    \caption{The area $\Omega$.  The pink area is one eighth of the whole shaded region.}
    \label{fig:areaofomega}
\end{figure}

To find the area of $\Omega$, consider the circles with equations $x^2+y^2=S^2$ and $(x+\abs{s})^2+y^2 = S^2$, and call them $C_1$ and $C_2$ respectively.  Then the region between these two circles, the line $y=x$ and the $x$-axis (as shown in Figure \ref{fig:areaofomega}) will be equal to $\frac{1}{8}\abs{\Omega}$.
Now, working in polar coordinates, $C_1$ and $C_2$ have equations $r=S$ and $r=-\abs{s}\cos{\theta} + (S^2-\abs{s}^2\sin^2{\theta})^{\frac{1}{2}}=:r_{\theta}$ respectively, so
\begin{equation*}
  \centering
  \abs{\Omega} = 8\int_0^{\frac{\pi}{4}} \int_{r_{\theta}}^S \! r \, \mathrm{d}r \, \mathrm{d}\theta.
\end{equation*}

We have
\begin{align*}
  \int_{r_{\theta}}^S \! r \, \mathrm{d}r &= \left[\frac{1}{2} r^2 \right]_{r_{\theta}}^S \\
   &= \frac{1}{2}S^2 - \frac{1}{2}\left(-\abs{s}\cos{\theta} + \left(S^2-\abs{s}^2\sin^2{\theta}\right)^{\frac{1}{2}}\right)^2 \\
   &= -\frac{1}{2}\abs{s}^2\left(\cos^2{\theta} - \sin^2{\theta}\right) + \abs{s}\cos{\theta}\left(S^2-\abs{s}^2\sin^2{\theta}\right)^{\frac{1}{2}} \\
   &= \abs{s}\cos{\theta}\left(S^2-\abs{s}^2\sin^2{\theta}\right)^{\frac{1}{2}} - \frac{1}{2}\abs{s}^2\cos{2\theta}.
\end{align*}
So, substituting this back into the integral for $\abs{\Omega}$, we have
\begin{align*}
  \abs{\Omega} &= 4\int_0^{\frac{\pi}{4}} \! 2\left(\abs{s}\cos{\theta}\left(S^2-\abs{s}^2\sin^2{\theta}\right)^{\frac{1}{2}} - \frac{1}{2}\abs{s}^2\cos{2\theta} \, \mathrm{d}\theta\right) \\
  &= -4\int_0^{\frac{\pi}{4}} \! \abs{s}^2\cos{2\theta} \, \mathrm{d}\theta + I_1 \\
  &= -4\left[\frac{1}{2}\abs{s}^2\sin{2\theta}\right]_0^{\frac{\pi}{4}} + I_1 \\
  &= -2\abs{s}^2 + I_1,
\end{align*}
where $I_1 = 8\int_0^{\frac{\pi}{4}} \! \abs{s}\cos{\theta}\left(S^2-\abs{s}^2\sin^2{\theta}\right)^{\frac{1}{2}} \, \mathrm{d}\theta$.  Now, into $I_1$ substitute $\sin{u}=\dfrac{\abs{s}}{S}\sin{\theta}$, so,
\begin{align*}
  I_1 &= 8S\abs{s}\int_0^{\sin^{-1}\left(\frac{\abs{s}}{\sqrt{2}S}\right)} \! \frac{S}{\abs{s}}\cos{u}\left(\cos^2{u}\right)^{\frac{1}{2}} \, \mathrm{d}u \\
  &= 8S^2 \int_0^{\sin^{-1}\left(\frac{\abs{s}}{\sqrt{2}S}\right)} \! \cos^2{u} \, \mathrm{d}u.
\end{align*}
\end{proof}

\subsection*{Estimating A}

We now aim to prove the following Proposition which gives an estimate for the sum $A$ associated with \eqref{eqn:firstmoment2}.  We start with Proposition \ref{prop:areaomega} and then use the lemmas from Section \ref{sec:prelimGI} to complete the proof.

\begin{prop} \label{prop:sumA}
For $A$ as defined in \eqref{eqn:firstmoment2},
  \begin{equation*}
  \centering
  A = \frac{\pi}{2}\zeta_i^{-1}(2)\left(8z_3-1\right)S^2 + \bgoh{S\ln{S}}.
\end{equation*}
\end{prop}

\begin{proof}
We begin by substituting our value for the area of $\Omega$ into $A$, which gives us
\begin{equation*}
  \centering
  A = \sum\limits_{\abs{s} \leq S} \dfrac{\phi_i(s)}{\abs{s}^4}I_1 - 2 \sum\limits_{\abs{s}\leq S}\dfrac{\phi_i(s)}{\abs{s}^2}.
\end{equation*}

We focus first on the second sum, using Lemmas \ref{lem:phii}, \ref{lem:sumr2} and \ref{lem:sumq>Q},
\begin{align}\label{eqn:A2}
  \sum\limits_{\abs{s}\leq S}\dfrac{\phi_i(s)}{\abs{s}^2} &= \sum\limits_{\substack{s\in \ZZp\\ \abs{s} \leq S}} \sum\limits_{\substack{d\in \ZZp\\ d|s}} \dfrac{\mu_i(d)}{\abs{d}^2} \nonumber \\
  &= \sum\limits_{\substack{d\in \ZZp\\ \abs{d} \leq S}} \dfrac{\mu_i(d)}{\abs{d}^2} \sum\limits_{\substack{s'\in \ZZp\\ \abs{s'}\leq \frac{S}{\abs{d}}}} 1 \nonumber \\
  &= \sum\limits_{\substack{d\in \ZZp\\ \abs{d} \leq S}} \dfrac{\mu_i(d)}{\abs{d}^2} \sum\limits_{k \leq \frac{S^2}{\abs{d}^2}} \dfrac{r_2(k)}{4} \nonumber \\
  &= \dfrac{1}{4}\sum\limits_{\substack{d\in \ZZp\\ \abs{d} \leq S}} \dfrac{\mu_i(d)}{\abs{d}^2} \left(\dfrac{\pi S^2}{\abs{d}^2} + \bigoh\left(\dfrac{S^{2\kappa}}{\abs{d}^{2\kappa}}\right)\right) \nonumber \\
  &= \dfrac{\pi}{4} S^2 \sum\limits_{\substack{d\in \ZZp\\ \abs{d} \leq S}} \dfrac{\mu_i(d)}{\abs{d}^4} + \bigoh\left(S^{2\kappa}\sum\limits_{\substack{d\in \ZZp\\ \abs{d} \leq S}} \dfrac{\mu_i(d)}{\abs{d}^{2+2\kappa}} \right) \nonumber \\
  &= \dfrac{\pi}{4}S^2\left(\zeta_i^{-1}(2)+\bigoh\left(\dfrac{1}{S^2}\right)\right) + \bigoh\left( S^{2\kappa}\left( \zeta_i^{-1}(1+\kappa) + \bigoh\left( \dfrac{1}{S^{\kappa}} \right) \right) \right) \nonumber \\
  &= \dfrac{\pi}{4}\zeta_i^{-1}(2)S^2 + \bigoh(S^{2\kappa}),
  \end{align}
where $\frac{1}{4}<\kappa<\frac{1}{2}$.

Now, before moving on to the first part of $A$, consider the sum $\sum\limits_{\abs{s}\leq S} \frac{\phi_i(s)}{\abs{s}^4}$ and apply Abel's Summation Formula with $x=S^4$, $f(t)=\frac{1}{t}$ and $a(n)=\sum\limits_{\abs{s}=n^{\frac{1}{4}}}\phi_i(s)$.  Then $A(t)=\sum\limits_{\abs{s}\leq t^{\frac{1}{4}}}\phi_i(s)$ and so, using Lemma \ref{lem:sumphii},
\begin{align*}
  \sum\limits_{\abs{s}\leq S} \frac{\phi_i(s)}{\abs{s}^4} &= \sum\limits_{n\leq S^4} \frac{a(n)}{n} \\
  &= \sum\limits_{\abs{s}\leq S}\phi_i(s)S^{-4} + \int_1^{S^4} \! A(t)t^{-2} \, \mathrm{d}t \\
  &= \left(z_1S^4 +\bgoh{S^{2+2\kappa}}\right)S^{-4} + \int_1^{S^4} \! \left(z_1t + \left(A(t)-z_1t\right)\right)t^{-2} \, \mathrm{d}t \\
  &= z_1 + \bgoh{S^{2\kappa-2}} + z_1 \int_1^{S^4} \! t^{-1} \, \mathrm{d}t + \int_1^{S^4} \! \left(A(t)-z_1t\right)t^{-2} \, \mathrm{d}t \\
  &= 4z_1\ln{S} + z_1 + \int_1^{\infty} \! \left(A(t)-z_1t\right)t^{-2} \, \mathrm{d}t - \int_{S^4}^{\infty} \! \left(A(t)-z_1t\right)t^{-2} \, \mathrm{d}t + \bgoh{S^{2\kappa-2}},
\end{align*}
where $z_1 = \frac{\pi}{8}\zeta_i^{-1}(2)$.
Define $z_2 := \int_1^{\infty} \! \left(A(t)-z_1t\right)t^{-2} \, \mathrm{d}t$ and note that it is absolutely convergent and so is well-defined.  Note also that we have
\begin{align*}
  \int_{S^4}^{\infty} \! \left(A(t)-z_1t\right)t^{-2} \, \mathrm{d}t & \ll \int_{S^4}^{\infty} \! \frac{1}{t^{\frac{3}{2}-\frac{\kappa}{2}}} \, \mathrm{d}t \\
  & \ll \left[t^{\frac{\kappa}{2} - \frac{1}{2}}\right]_{S^4}^{\infty} \\
  & \ll S^{2\kappa-2}.
\end{align*}
Thus,
\begin{equation}\label{eqn:sumphii4}
  \centering
  \sum\limits_{\abs{s}\leq S} \frac{\phi_i(s)}{\abs{s}^4} = 4z_1\ln{S} + \left(z_1 + z_2\right) + \bgoh{S^{2\kappa-2}}.
\end{equation}

Now, returning to $A$, we need to estimate $\sum\limits_{\abs{s}\leq S} \frac{\phi_i(s)}{\abs{s}^4}I_1$.  We apply Abel's Summation Formula with $x=S^4$, $f(t)=I_1$, and $a(n)=\sum\limits_{\abs{s}=n^{\frac{1}{4}}}\phi_i(s) n^{-1}$.  Then by \eqref{eqn:sumphii4},
\begin{equation*}
A(t)=\sum\limits_{\abs{s}\leq t^{\frac{1}{4}}}\frac{\phi_i(s)}{\abs{s}^4} = z_1\ln{t} + \left(z_1 + z_2 \right) +\bgoh{t^{\frac{\kappa}{2}-\frac{1}{2}}}.
\end{equation*}
Also, by the Fundamental Theorem of Calculus,
\begin{align*}
  f'(t) &= 8S^2\cos^2\left(\sin^{-1}\left(\frac{t^\frac{1}{4}}{\sqrt{2}S}\right)\right) \frac{d}{dt} \left(\sin^{-1}\left(\frac{t^{\frac{1}{4}}}{\sqrt{2}S}\right)\right) \\
  &= 8S^2\left(1-\frac{t^{\frac{1}{2}}}{2S^2}\right)\frac{t^{-\frac{3}{4}}}{4\sqrt{2}S} \left(1-\frac{t^{\frac{1}{2}}}{2S^2}\right)^{-\frac{1}{2}} \\
  &= \sqrt{2}St^{-\frac{3}{4}}\left(1-\frac{t^{\frac{1}{2}}}{2S^2}\right)^{\frac{1}{2}}.
\end{align*}
We then have,
\begin{align}\label{eqn:X1+X2}
  \int_{1}^{S^4} \! A(t)f'(t) \, \mathrm{d}t &= \int_{1}^{S^4} \! \sqrt{2}St^{-\frac{3}{4}} \left(z_1\ln{t} + \left(z_1 + z_2 \right) + \bgoh{t^{\frac{\kappa}{2}-\frac{1}{2}}}\right)\left(1-\frac{t^{\frac{1}{2}}}{2S^2}\right)^{\frac{1}{2}} \, \mathrm{d}t \nonumber \\
  &= \sqrt{2}z_1 S \int_1^{S^4} \! t^{-\frac{3}{4}}\ln{t} \left(1-\frac{t^{\frac{1}{2}}}{2S^2}\right)^{\frac{1}{2}} \, \mathrm{d}t \nonumber \\
  & \qquad+ \sqrt{2}\left(z_1+z_2\right)S\int_1^{S^4}t^{-\frac{3}{4}}\left(1-\frac{t^{\frac{1}{2}}}{2S^2}\right)^{\frac{1}{2}} \, \mathrm{d}t \nonumber \\
  & \qquad+ \bgoh{S\int_1^{S^4}t^{\frac{\kappa}{2}-\frac{5}{4}}\left(1-\frac{t^{\frac{1}{2}}}{2S^2}\right)^{\frac{1}{2}} \, \mathrm{d}t} \nonumber \\
  &= X_1+X_2+\bgoh{S}.
\end{align}

Now, substituting $\sin{\theta}=\frac{t^{\frac{1}{4}}}{\sqrt{2}S}$,
\begin{align*}
  \int_1^{S^4}t^{-\frac{3}{4}}\left(1-\frac{t^{\frac{1}{2}}}{2S^2}\right)^{\frac{1}{2}} \, \mathrm{d}t &= 4\sqrt{2}S\int_{\sin^{-1}(\frac{1}{\sqrt{2}S})}^{\frac{\pi}{4}} \! \cos^2{\theta} \, \mathrm{d}\theta \\
  %&= 2\sqrt{2}S\left[\sin{\theta}\cos{\theta}+\theta\right]_{\sin^{-1}(\frac{1}{\sqrt{2}S})}^{\frac{\pi}{4}} \\
  %&= 2\sqrt{2}S\left(\left(\frac{1}{2}+\frac{\pi}{4}\right)-\left( \frac{1}{\sqrt{2}S}\cos{\sin^{-1}(\frac{1}{\sqrt{2}S})} + \sin^{-1}(\frac{1}{\sqrt{2}S}) \right) \right) \\
  &= \sqrt{2}S\left(\frac{\pi}{2}+1\right)+\bgoh{1}
\end{align*}
and so \begin{equation}\label{eqn:X2}
         \centering
         X_2 = (z_1+z_2)(\pi + 2)S^2+\bgoh{S}.
       \end{equation}

Finally, letting $u=\frac{t^{\frac{1}{4}}}{\sqrt{2}S}$,
\begin{align*}
  \int_1^{S^4} \! t^{-\frac{3}{4}}\ln{t} \left(1-\frac{t^{\frac{1}{2}}}{2S^2}\right)^{\frac{1}{2}} \, \mathrm{d}t &= 16\sqrt{2}S \int_{\frac{1}{\sqrt{2}S}}^{\frac{1}{\sqrt{2}}} \! \ln(\sqrt{2}uS) \left(1-u^2\right)^{\frac{1}{2}} \, \mathrm{d}u \\
  &= 16\sqrt{2}S\ln{S} \int_{\frac{1}{\sqrt{2}S}}^{\frac{1}{\sqrt{2}}} \! \left(1-u^2\right)^{\frac{1}{2}} \, \mathrm{d}u \\
  & \qquad+ 16\sqrt{2}S \int_{\frac{1}{\sqrt{2}S}}^{\frac{1}{\sqrt{2}}} \! \ln(\sqrt{2}u) \left(1-u^2\right)^{\frac{1}{2}} \, \mathrm{d}u \\
  &= 16\sqrt{2}S\ln{S} \left( \int_{0}^{\frac{1}{\sqrt{2}}} \! \left(1-u^2\right)^{\frac{1}{2}} \, \mathrm{d}u - \int_{0}^{\frac{1}{\sqrt{2}S}} \! \left(1-u^2\right)^{\frac{1}{2}} \, \mathrm{d}u \right)\\
  & \qquad+ 16\sqrt{2}S\int_{\frac{1}{\sqrt{2}S}}^{\frac{1}{\sqrt{2}}} \! \ln(\sqrt{2}u) \left(1-u^2\right)^{\frac{1}{2}} \, \mathrm{d}u  \\
  &= \sqrt{2}\left(2\pi +4\right)S\ln{S} + \bgoh{\ln{S}} \\
  & \qquad+ 16\sqrt{2}S\int_{0}^{\frac{1}{\sqrt{2}}} \! \ln(\sqrt{2}u) \left(1-u^2\right)^{\frac{1}{2}} \, \mathrm{d}u \\
  & \qquad- 16\sqrt{2}S\int_{0}^{\frac{1}{\sqrt{2}S}} \! \ln(\sqrt{2}u) \left(1-u^2\right)^{\frac{1}{2}} \, \mathrm{d}u \\
  &= \sqrt{2}\left(2\pi +4\right)S\ln{S} - 16\sqrt{2}z_3S + \bgoh{\ln{S}},
\end{align*}
where $z_3 = -\int_{0}^{\frac{1}{\sqrt{2}}} \! \ln(\sqrt{2}u) \left(1-u^2\right)^{\frac{1}{2}} \, \mathrm{d}u > 0$.  Thus,
\begin{equation}\label{eqn:X1}
  \centering
  X_1 = (4\pi +8)z_1S^2\ln{S} - 32z_1z_3S^2 + \bgoh{S\ln{S}}.
\end{equation}
Now \eqref{eqn:X1+X2}, \eqref{eqn:X2} and \eqref{eqn:X1} give us
\begin{align*}
  \sum\limits_{\abs{s}\leq S} \frac{\phi_i(s)}{\abs{s}^4}I_1 &= A(S^4)f(S^4)-\int_1^{S^4} \! A(t)f'(t) \, \mathrm{d}t \\
  &= \left(4z_1\ln{S} + (z_1+z_2) + \bgoh{S^{2\kappa -2}} \right)(\pi+2)S^2 - X_1 - X_2 + \bgoh{S} \\
  &= 32z_1z_3S^2 + \bgoh{S\ln{S}}.
\end{align*}
This, together with \eqref{eqn:A2}, gives our estimate for $A$,
\begin{equation*}
  \centering
  A = \frac{\pi}{2}\zeta_i^{-1}(2)\left(8z_3-1\right)S^2 + \bgoh{S\ln{S}}.
\end{equation*}
\end{proof}

\subsection*{Estimating B}

The last thing we need is an estimate for the sum $B$ associated with \eqref{eqn:firstmoment2}. This will be achieved by splitting the sum over dyadic annuli.

\begin{prop}
For $B$ as defined in \eqref{eqn:firstmoment2},
\begin{equation*}
  \centering
  B \ll S^{1+\epsilon}
\end{equation*}
for all $\epsilon >0$.
\end{prop}

\begin{proof}
  Clearly $\abs{\partial\Omega} \ll S$ so, substituting this into $B$,
 \begin{align*}
   B &= \sum\limits_{\abs{s}\leq S} \dfrac{\abs{\partial\Omega}}{\abs{s}^{2-\epsilon}} \\
   & \ll S \sum\limits_{\abs{s}\leq S} \dfrac{1}{\abs{s}^{2-\epsilon}} \\
   & \ll S \sum\limits_{k\leq \log_2{S}} \sum\limits_{2^{k-1} \leq \abs{s} < 2^k} \dfrac{1}{\abs{s}^{2-\epsilon}}\\
   & \ll S \sum\limits_{k \leq \log_2{S}} \dfrac{1}{2^{k(2-\epsilon)}} \sum\limits_{2^{k-1} \leq \abs{s} < 2^k} 1 \\
   & \ll S \sum\limits_{k \leq \log_2{S}} \left(2^k\right)^{\epsilon} \\
   & \ll S^{1+\epsilon},
 \end{align*}
using the fact that
\begin{equation*}
\sum\limits_{2^{k-1} \leq \abs{s} < 2^k} 1 \asymp 2^{2k}.
\end{equation*}
\end{proof}

Finally, putting together our estimates for A and B, we have
\begin{equation*}
  \centering
  M_{1,I_2}(S) = \pi \zeta_i^{-1}(2)\left(8z_3 -1\right)S^2 + \mathcal{O}_{\epsilon}(S^{1+\epsilon}).
\end{equation*}
Note that $z_3 \approx 0.68644 >\frac{1}{2}$, so $8z_3-1$ is positive.

\section*{Acknowledgements}
I wish to thank Alan Haynes for his help and advice throughout this project, in particular with the proof of Theorem \ref{thm:latticepointcount}.  I am also grateful to Christopher Hughes and Sanju Velani for their comments and suggestions on the writing of this paper.

\end{document}